\documentclass{amsart}
\usepackage{amsmath,amsthm}
\usepackage{amsfonts,amssymb}
\usepackage{accents}
\usepackage{enumerate}
\usepackage{accents,color}
\usepackage{graphicx}
\newcommand{\lvt}{\left|\kern-1.35pt\left|\kern-1.3pt\left|}
\newcommand{\rvt}{\right|\kern-1.3pt\right|\kern-1.35pt\right|}

\newtheorem{thm}{Theorem}[section]
\newtheorem{cor}[thm]{Corollary}
\newtheorem{lem}[thm]{Lemma}
\newtheorem{prop}[thm]{Proposition}

\newtheorem{defn}[thm]{Definition}

\theoremstyle{remark}

 \def\d{\mathrm{d}}
 \def\i{\mathrm{i}}
 \def\e{\mathrm{e}}

 \def\fG{{\mathfrak G}}
  
 \def\a{{\alpha}}
 \def\b{{\beta}}

 \def\t{{\theta}}
 \def\l{{\lambda}}

 \def\CM{{\mathcal M}}

 \def\CC{{\mathbb C}}
 
 \def\NN{{\mathbb N}}

 \def\RR{{\mathbb R}}
 
 \def\TT{{\mathbb T}}

 \def\ZZ{{\mathbb Z}}

\def\lla{\langle{\kern-2.5pt}\langle}      
\def\rra{\rangle{\kern-2.5pt}\rangle}

\newcommand{\wh}{\widehat}

\def\fD{{\mathfrak D}}
\def\fG{{\mathfrak G}}

\def\f{\frac}

\graphicspath{{./}}
\begin{document}

\title[$\ell^1$-summability and Fourier series of B-splines]
{$\ell^1$-summability and Fourier series of B-splines with
  respect to their knots}

\author{Martin Buhmann}
\address{Justus-Liebig University, Lehrstuhl Numerische Mathematik, 35392 Giessen, Germany}
\email{Martin.Buhmann@math.uni-giessen.de}
\author{Janin J\"ager}
\address{Justus-Liebig University, Lehrstuhl Numerische Mathematik, 35392 Giessen, Germany}
\email{janin.jaeger@math.uni-giessen.de}
\author{Yuan~Xu}
\address{Department of Mathematics, University of Oregon, Eugene, 
OR 97403--1222, USA}
\email{yuan@uoregon.edu} 
\thanks{The first author was funded by the Deutsche Forschungsgemeinschaft 
(DFG-German research foundation)Projektnummer: 461449252. The third author thanks the Alexander von Humboldt Foundation for an AvH award that 
supports his visit to Justus-Liebig University, during which the work was carried out; he was partially supported by 
Simons Foundation Grant \#849676.}
\date{\today}  
 
\begin{abstract} 
We study the $\ell^1$-summability of functions in the $d$-dimensional
torus $\TT^d$ and so-called $\ell^1$-invariant functions. Those are functions
on the torus whose Fourier coefficients depend only on the
$\ell^1$-norm of their indices. Such functions are characterized as 
divided differences that have $\cos \t_1,\ldots,\cos\t_d$ as knots for
$(\t_1\,\ldots, \t_d) \in \TT^d$. It leads us to consider the $d$-dimensional 
Fourier series of univariate B-splines with respect to its knots, which turns 
out to enjoy a simple bi-orthogonality that can be used to obtain an orthogonal
series of the B-spline function.   
\end{abstract} 
 
\maketitle 

\section{Introduction}
\setcounter{equation}{0}

We consider a problem originating from the $\ell^1$-summability of multivariate Fourier series. Let $f$ be a
$2\pi$-periodic function in $L^2(\TT^d)$ and let $\hat f_\a$ be the Fourier coefficient of $f$ with the multi-index 
$\a = (\a_1,\ldots, \a_d) \in \ZZ^d$. Let 
$$
S_n^{(1)} (f;\t)= \sum_{|\a| \le n} \hat f_\a \e^{\i \a\cdot \t}, \quad \t \in \TT^d, \quad n \in \NN_0,
$$
be the $n$-th $\ell^1$-partial sum of its Fourier series, where $|\a|=|\a|_1 = |\a_1|+ \cdots + |\a_d|$, so that the 
summation is over indices in the $\ell^1$-ball of radius $n$. The $\ell^1$-summability has been studied in 
\cite{BX1, N, SV, FW1, FW2, FW3} and it is closely related to the summability of Fourier series in orthogonal 
polynomials on the cube \cite{X95}. The Dirichlet kernel of $S_n^{(1)}(f)$ turns out to be a divided difference to
be defined below in the form
$$
    D_{n}^{(1)}(\t) = [\cos \t_1, \ldots, \cos \t_d] G_{n,d}, \qquad \t \in \TT^d,
$$
where $G_{n,d}$ is a function of one variable as shown in \cite{BX1,X95} (see \eqref{eq:Gnd=} in the next section).
The divided difference can be written as an integral with a Peano kernel; in particular, for a $(d-1)$-times differentiable function $F: [-1,1]\to 
\CC$, 
$$
    [\cos \t_1,\ldots,\cos \t_d] F = \int_{-1}^1 F^{(d-1)}(u) M_{d-1}(u| \cos \t_1,\ldots,\cos \t_d) \d u,
$$ 
where $u \mapsto M_{d-1}(u| \cos \t_1,\ldots,\cos \t_d)$ is the B-spline function, which is a piecewise polynomial
function in $C^{d-2}([-1,1])$ with $\cos \t_1,\ldots, \cos \t_d$ as its knots (see the next section for its definition). 
Motivated by the $\ell^1$-summability and functions defined by the above divided difference, we call a function 
$\ell^1$-invariant if $\hat f_\a = \hat f_\b$ whenever $|\a| = |\b|$ and study properties of such functions. 

Our analysis is partially motivated by the study in \cite{BX1}, where the $\ell^1$-summability of the Fourier 
transform in $\RR^d$, defined by 
$$
 R_{\rho,\d}^{(1)}(f; x) = \int_{|v|_1 \le \rho} \hat f(v) \e^{\i v \cdot x} \d v, \quad x \in \RR^d \quad\hbox{and}\quad \rho \ge 0,
$$
is treated and its associated Dirichlet kernel is shown to be given as a divided difference,
$$
     \fD_{\rho, \d}(x) =  \left[x_1^2,\ldots,x_d^2\right] \fG_{\rho, d}, \qquad x \in \RR^d,
$$
where $\fG_{\rho,d}$ is a function of one variable. The Fourier transform of the B-spline function 
$x \mapsto M_{d-1}(u| x_1^2,\ldots, x_d^2)$, considered as a function of its knots, is analyzed in \cite{BX1}, 
which turns out to enjoy a rich structure and provides necessary tools for studying the class of $\ell^1$-invariant
functions $f(\|\cdot\|_1)$ defined on $\RR^d$. In particular, it leads to a characterization of $f: \RR_+ \to 
\RR$ so that $x \mapsto f(\|\cdot\|_1)$ is a positive definite function on $\RR^d$. 

We will show that the $\ell^1$-invariant functions on the torus are all given by divided differences with
knots $\cos \t_1, \ldots, \cos \t_d$ and we will study the Fourier series of the $B$-spline $\t \mapsto 
M_{d-1}(u| \cos \t_1,\ldots, \cos \t_d)$, which is $\ell^1$-invariant. While the Fourier transform of the B-spline 
$x \mapsto M_{d-1}(u| x_1^2,\ldots, x_d^2)$ on $\RR^d$ satisfies an integral recursive relation in dimension $d$, the
Fourier coefficients of the $B$-spline $x \mapsto M_{d-1}(u| \cos \t_1,\ldots, \cos \t_d)$ on $\TT^d$ satisfy
a somewhat surprising biorthogonal relation with a family of polynomials. Let $m_{n,d}$ denote the Fourier 
coefficients of the $B$-spline function with the index $|\a| = n$. Then there is a sequence of polynomials
$h_{n,d}$, given in terms of the Gegenbauer polynomials, such that $\{m_{n,d}: n \in \NN_0\}$ and
$\{h_{n,d}: n \in \NN_0\}$ are biorthogonal in the sense that 
$$
\int_{-1}^1 m_{n,d}(u) h_{\ell,d}(u)\,\d u = \delta_{n,\ell}, \qquad n, \ell = 0 ,1, 2, \ldots.
$$
This orthogonal relation can be used to derive the Fourier orthogonal series of the B-spline function in the 
Gegenbauer polynomials explicitly; the first term of the series gives, in particular, that 
$$
  \frac{1}{(2\pi)^d} \int_{\TT^d} M_{d-1}(u| \cos \t_1,\ldots,\cos \t_d)\,\d \t = 
      \frac{\Gamma(\frac{d+1}2)}{\sqrt{\pi}\Gamma(\f{d}{2})(d-1)!} (1-u^2)_+^{\f{d-2}{2}},
$$
an identity that appears to be new. We will also give a characterization of $\ell^1$-invariant functions that are 
either positive definite or strictly positive definite on $\TT^d$. 

The paper is organized as follows. We recall the definition and basic properties of $\ell^1$-summability in the
next section and establish several necessary identities. The Fourier orthogonal series of the B-spline with 
respect to its knot is given in the third section. The positive definite functions of $\ell^1$-invariant functions are 
discussed in the fourth section. 

\section{$\ell^1$-summability on $\TT^d$}
\setcounter{equation}{0}

Let $f$ be a $2\pi$-periodic function defined on $\TT^d$. If $f\in L^2(\TT^d)$, then the Fourier series of $f$ is 
 defined by
$$
   f(x) = \sum_{\a \in \ZZ^d} \hat f_\a \e^{\i \a \cdot x},\; x\in\TT^d, \quad \hbox{with} 
       \quad \hat f_\a = \frac{1}{(2\pi)^d} \int_{\TT^d} f(y) \e^{\i \a \cdot y}\,\d y,\;\alpha\in\ZZ^d.
$$ 
We study the class of periodic functions that we call $\ell^1$-invariant. 

\begin{defn} \label{defn:l1-invariant}
A function $f: \TT^d \to  \RR$ is called $\ell^1$-invariant if 
$$
   \hat f_\a = \hat f_{\b} \quad \hbox{whevever $|\a| = |\b|$ for $\a, \b \in \ZZ^d$}. 
$$
We denote the Fourier coefficient $\hat f_\a$ of such a function by $\hat f_{|\a|}$. 
\end{defn}

If $f$ is $\ell^1$-invariant, then its Fourier series is of the form
\begin{equation} \label{eq:l1-function}
       f(x) = \sum_{n=0}^\infty \hat f_n E_n(x), \qquad E_n(x) = \sum_{|\a| = n} \e^{\i \a \cdot x}, \quad x\in\TT^d,
\end{equation}
where $|\a|$ is the $\ell^1$-norm, that is $|\a|:= |\a_1| + \cdots + |\a_d|$, of $\a \in \ZZ^d$. 

A function $f$ on $\TT^d$ is called $\ell^1$-summable if its partial sum $S_n^{(1)} f$ over the expanding 
$\ell^1$-ball, defined by 
$$
   S_n^{(1)} f(x) =  \sum_{|\a| \le n} \hat f_\a \e^{\i \a \cdot x},\qquad x\in\TT^d,
$$
converges to $f$. The partial sum can be written as an integral operator 
$$
 S_n^{(1)} f(x) = \frac{1}{(2\pi)^d} \int_{\TT^d} f(y) D_{n,d}(x-y)\,\d y = f*D_{n,d}(x), 
 $$
where the kernel $D_{n,d}$ is the analog of the Dirichlet kernel defined by 
$$
      D_{n,d}(y):= \sum_{|\a|\le n} \e^{\i \a \cdot y},\qquad y\in\TT^d.
$$
It is shown in \cite{BX1, X95} that the kernel $D_{n,d}$
can be written as a divided difference  
$$
  D_{n,d}(x) = [\cos x_1,\ldots,\cos x_d] G_{n,d},
$$
where $G_{n,d}$ is a univariant function defined by 
\begin{equation}\label{eq:Gnd=}
  G_{n,d}(\cos \t) =  (-1)^{\lfloor \frac{d-1}2 \rfloor} 2 \cos\tfrac\t 2(\sin\t)^{d-2}
\times \begin{cases} \cos(n+\tfrac 12)\t & \hbox{for $d$ even}, \\
      \sin(n+\tfrac12)\t & \hbox{for $d$ odd}. \end{cases}
\end{equation}

We briefly recall the notion of a divided difference of a function that is at least continuous. Let $f$ be a real or 
complex function on $\RR$, and let $m \in \NN_0$. The $m$-th divided difference of $f$ at the (pairwise distinct) 
knots, $x_0, x_1, \ldots, x_m$ in $\RR$ is defined inductively as 
$$
  [x_0]\,f = f(x_0) \quad\hbox{and}\quad [x_0,\ldots,x_m]\,f =
      \frac{[x_0,\ldots,x_{m-1}]\,f - [x_1,\ldots,x_m]\,f }{x_0 - x_m}.
$$
The divided difference is a symmetric function of the knots. The knots of the divided difference may coalesce. 
In particular, if all knots coalesce and if the function is sufficiently differentiable, then the divided difference collapses to 
\begin{equation}\label{eq:divder}
  [x_0,\ldots,x_m] f = \frac{f^{(m)}(x_0)}{m!} \quad \hbox{if $x_0 = x_1 = \cdots = x_m$}.
\end{equation}
Our analysis depends heavily on an
integral representation of the divided difference, for which we need the definition of B-spline. 
For $x_0 < \cdots < x_m$, the B-spline of order $m$ with knots $x_0, \ldots, x_m$ is defined by
$$
\RR \ni u \quad\to\quad M_m(u | x_0,\ldots,x_m) = [x_0,\ldots,x_m]\left \{ \frac{(\,\cdot\, - u)_+^{m-1}} {(m-1)!}\right\}.
$$
The B-spline vanishes outside the interval $(x_0,x_n)$ and it is strictly positive on the interval itself, and 
$$
     \int_\RR M_m(u|x_0,\ldots,x_m)\;\,\d u = \frac{1}{m!}.
$$
For better reference, we state the integral representation of the divided difference as a lemma. 

\begin{lem}
Let $f :\RR \to \CC$ be $m$-times continuously differentiable. Then
$$
[x_0,\ldots,x_m]\,f = \int_\RR f^{(m)}(u)M_m(u | x_0,\ldots,x_m)\, \d u.
$$
\end{lem}
We shall also need the B-splines' recurrence relation (Powell, 1982)
\begin{align*}
M_{m+1}(u| x_0,\ldots,x_m,x_{m+1}) \, &  = \frac{(u-x_0)M_m(u | x_0,\ldots,x_m)}{x_{m+1}-x_0}\\
&+ \frac{(x_{m+1}-u)M_m(u | x_1,\ldots,x_{m+1})}{x_{m+1}-x_0}.
\end{align*}

We first write the function $E_n$ as an integral against the B-spline. We will need the Gegenbauer polynomials,
which are orthogonal polynomials with respect to the weight function
$$
  w_\l(t) = (1-t^2)^{\l-\f12}, \quad \l > -\tfrac 12,
$$
on the interval $[-1,1]$. The Gegenbauer polynomial of degree $n$ is denoted by $C_n^\l$ and normalized by
$C_n^\l(1) = \frac{(2\l)_n}{n!}$. The Gegenbauer polynomials satisfy the orthogonality 
$$
  c_\l \int_{-1}^1 C_n^\l(t) C_m^\l(t) w_\l(t) \,\d t =  \frac{\l}{n+\l} C_n^\l(1) \delta_{n,m}, \qquad n, m \in \NN_0,
$$
where $c_\l$ is a constant so that $c_\l \int_{-1}^1 w_\l(t) \,\d t = 1$ and $(a)_n = a(a+1)\cdots (a+n-1)$ 
denotes the Pochhammer symbol. For convenience, we also define
$$
    Z_n^\l(t):= \frac{n+\l}{\l} C_n^\l(t).
$$
The generating function of the Gegenbauer polynomials is given by 
$$
   \frac{1}{(1- 2r t + r^2)^{\l}} = \sum_{n=0}^\infty C_n^\l (t) r^n, \qquad 0 \le r < 1.
$$
Throughout this paper we define $C_n^\l(t) =0$ whenever $n < 0$. 

\begin{lem}
For $\theta = (\t_1,\ldots,\t_d)\in \TT^d$, the function $E_n$ satisfies 
\begin{align} \label{eq:En-hnd}
  E_n(\theta) \, &= [\cos \t_1, \ldots, \cos \t_d] H_{n,d} \\
       & = \int_{-1}^1 h_{n,d}(u) M_{d-1}(u | \cos \t_1,\ldots, \cos \t_d) \,\d u, \notag
\end{align}
where $H_{n,d}$ and $h_{n,d}$ are defined by 
\begin{align} \label{eq:Hnd=}
  H_{n,d}(\cos \t) =  2 (-1)^{\lfloor \f{d-1}{2}\rfloor} (\sin \t)^{d-1} \times 
      \begin{cases} - \sin (n \t) & \hbox{for $d$ even}, \\ \cos (n \t) & \hbox{for $d$ odd},  \end{cases}
\end{align}
and $h_{n,d}$ is a polynomial of degree $n$ given by 
\begin{align} \label{eq:hnd=}
   h_{n,d}(u) \,& = (d-1)! \sum_{j=0}^d (-1)^j \binom{d}{j} C_{n-2j}^d(u) \\
   & = (d-1)! \sum_{j=0}^{d-1}(-1)^j \binom{d-1}{j}Z_{n-2j}^{d-1}(u). \notag
\end{align}
\end{lem} 

\begin{proof}
By its definition, $E_n(\t) = D_n(\t) - D_{n-1}(\t)$, so that $E_n(\t)$ is a divided difference of 
$H_{n,d} = G_{n,d} - G_{n-1,d}$, from which \eqref{eq:En-hnd} follows readily with $h_{n,d} = H_{n,d}^{(d-1)}$. 
Moreover, the identity \eqref{eq:Hnd=} is an immediate consequence of \eqref{eq:Gnd=} and elementary 
trigonometric identities. Now, it is shown in \cite{BX1} that  
$$
   g_{n,d}(t) = G_{n,d}^{(d-1)}(t) = f_{n,d}(t) + f_{n-1,d}(t), 
$$
where $f_{n,d}$ is given in terms of the Gegenbauer polynomials by 
$$
  f_{n,d} (t) = (d-1)! \sum_{j=0}^{d-1} (-1)^j \binom{d-1}{j} C_{n-2j}^{d}(t).
$$
Using the relation \cite[(4.7.29)]{Sz} 
$$
   (n+\l) C_n^{\l}(t) = \l \left[ C_n^{\l+1}(t) - C_{n-2}^{\l+1}(t) \right]
$$
with $\l = d-1$, we then obtain 
$$
  h_{n,d} = f_{n,d} - f_{n-2,d} = (d-1)!\sum_{j=0}^d (-1)^j  \binom{d-1}{j} \left[C_{n-2j}^d(t) - C_{n-2j-2}^{d}(t)\right] 
$$
which is the second expression of $h_{n,d}$ in \eqref{eq:hnd=} by recursion with $Z_n^d$. Furthermore, the first identity in \eqref{eq:hnd=}
follows from $\binom{d-1}{j} + \binom{d-1}{j-1} = \binom{d}{j}$ and
$$
     h_{n,d} = f_{n,d} - f_{n-2,d} = (d-1)!\sum_{j=0}^d (-1)^j  \left( \binom{d-1}{j} + \binom{d-1}{j-1}\right) C_{n-2j}^d,
$$
where we define for convenience $\binom{d-1}{m} = 0$ if $m = -1$ or $m = d$. 
\end{proof}

Let $N_d(n) = \# \{\a \in \NN_0^d: |\a| =n\}$ be the cardinality of the set $\{\a: |\a| =n\}$. Then, 
$N_{d}(n) =  E_n(0)$. As a consequence of the identities \eqref{eq:En-hnd} and \eqref{eq:hnd=}, we obtain
$$
    N_d(n) = E_n(0) = \frac{h_{n,d}(1)}{(d-1)!} = \sum_{j=0}^d \frac{(-d)_j (2d)_{n-2j}}{j! (n-2j)!},
$$ 
where $(a)_n = a(a+1)\cdots(a+n-1)$ is the Pochhammer symbol. 
The last sum can be written as a hypergeometric ${}_3F_2$ function evaluated at $1$, 
but the series is 
not balanced so it does not have a closed-form formula. The first values of $N(n,d)$ are given below
$$
  N_2(n) = 4n, \quad N_3(n) =  2 n^2+1, \quad N_4(n) =  \frac4 9 n (n^2+2).
$$

The function $h_{n,d}$ satisfies a generating function identity which we state as the following

\begin{lem}
Let $0 \le r < 1$. Then
\begin{equation}\label{eq:generat-hnd}
  (d-1)! \frac{(1-r^2)^d}{(1-2 r u + r^2)^d}  = \sum_{n= 0}^\infty h_{n,d}(u) r^n. 
\end{equation}
\end{lem}

\begin{proof}
By the explicit formula of $h_{n,d}$, we obtain
\begin{align*}
  \frac{1}{(d-1)!} \sum_{n=0}^\infty h_{n,d}(u) r^n 
       \, & = \sum_{n=0}^\infty  \sum_{j=0}^d (-1)^j \binom{d}{j} C_{n-2j}^d(u) r^n \\
       & =   \sum_{j=0}^d (-1)^j\binom{d}{j} \sum_{n=2j}^\infty C_{n-2j}^d(u) r^{n-2j} r^{2j} \\
       & =   \sum_{j=0}^d (-1)^j \binom{d}{j} r^{2j} \sum_{n= 0}^\infty C_n^d(u) r^n \\
       & = (1-r^2)^d \frac{1}{(1-2 r u+r^2)^d},
\end{align*}
where we have used the generating function of the Gegenbauer polynomials. 
\end{proof}

Our next result is of interest in itself, which gives an explicit formula for the divided difference of the function 
$$
     P_r(t): = \frac{1}{1-2 r t +r^2}, \quad 0 \le r < 1, \quad t \in [-1,1].
$$

\begin{prop}
For $0 \le r < 1$, 
\begin{align} \label{eq: Md-Poisson}
 (d-1)! \int_{-1}^1  \frac{1}{(1-2 r u + r^2)^d} & M_{d-1}(u | \cos \t_1,\ldots,\cos \t_d) \,\d u  \\
      & = \frac{1}{\prod_{i=1}^d (1-2 r \cos \t_i + r^2)}.  \notag
\end{align}
In particular, 
\begin{equation} \label{eq: diff-Poisson}
   [\cos \t_1, \ldots, \cos \t_d] P_r =   \frac{(2r)^{d-1}}{\prod_{i=1}^d (1-2 r \cos \t_i + r^2)}.
\end{equation}
\end{prop}

\begin{proof}
We start with the elementary identity 
$$
   \frac{1-r^2}{1-2r \cos \phi+r^2} = \sum_{n=0}^\infty r^{n} \e^{\i n \phi}. 
$$
Reorganizing the $d$-fold product of this identity and setting $\t = (\t_1,\ldots, \t_d)$ as above, we obtain the equalities
\begin{align} \label{eq:prod-poisson}
   \frac{(1-r^2)^d} {\prod_{i=1}^d (1-2 r \cos \t_i + r^2)} \,& = \sum_{n=0}^\infty r^n \sum_{|\a| = n} \e^{\i \a \cdot \t} \\
       & = \sum_{n=0}^\infty r^n E_n(\t) =  \sum_{n=0}^\infty r^n [\cos \t_1, \ldots, \cos \t_d] H_{n,d} \notag \\
       & = \sum_{n=0}^\infty r^n  \int_{-1}^1 h_{n,d}(u) M_{d-1} (u |\cos \t_1,\ldots, \cos \t_d) \,\d u, \notag
\end{align}
from which the identity \eqref{eq: Md-Poisson} follows from the generating function of $h_{n,d}$. Now, it is easy to verify that 
$$
   P_r^{(d-1)} (u) = \frac{(d-1)! (2r)^{d-1}}{(1- 2 r u + r^2)^d},
$$
so that the left-hand side of \eqref{eq: Md-Poisson} can be identified with the divided difference of $P_r$, 
which gives \eqref{eq: diff-Poisson}.
\end{proof} 

\section{Fourier series of B-spline with respect to its knots} 
\setcounter{equation}{0}

As a function of its knots, the B-spline function is a periodic function on $\TT^d$,  
$$
\TT^d \ni \t \mapsto M_{d-1}(u| \cos \t_1,\ldots,\cos\t_d) \in \RR,
$$
for each $u$, and we also define for convenience 
$$
     \CM_d(\a; \t) := M_{d-1}(\cos \alpha | \cos \t_1,\ldots,\cos\t_d). 
$$   
Studying this case is sufficient since for  $|u| \ge 1$,  $ M_{d-1}(u| \cos \t_1,\ldots,\cos\t_d) = 0$ by definition, for all $\t \in \TT^d$. We first show that it is an
integrable function on $\TT^d$. 

\begin{prop} 
For $u \in (-1,1)$, the function $\t \mapsto M_{d-1}(u| \cos \t_1,\ldots,\cos\t_d)$ is in $L^1(\TT^d)$. 
\end{prop}

\begin{proof}
Let $u = \cos \a$ for $0 < \a < \pi$ be fixed. Since the function $\CM_d(\a;\cdot)$ is obvious even in each of 
its variables, we only need to consider $\t \in [0, \pi]^d$. Furthermore, the divided difference is a symmetric 
function of its knots, the function $\CM_{d}(\a;\cdot)$ is a symmetric function and it is nonnegative, so 
we only need to show that it is an $L^1$ function on the domain
$$
 \triangle_d = \{ \t = (\t_1,\ldots,\t_d) \in \TT^d:  0 \le \t_d \le \t_{d-1}\le \cdots \le \t_1 \le \pi\}.
$$
Indeed, the above consideration leads readily to 
$$
  \int_{\TT^d} M_{d-1}(u| \cos \t_1,\ldots,\cos\t_d)\,\d \t = 2^d d!  \int_{\triangle_d} 
     M_{d-1}(u| \cos \t_1,\ldots,\cos\t_d)\,\d \t.
$$

We start with the case $d=2$; the univariate case is trivial by continuity and
 compact support. On the domain $\triangle_2$, the function is given by
$$
  \CM_{2}(\a; \t) = \frac{\chi_{[\t_1,\t_2]} (\a)}{\cos \t_2 - \cos \t_1}
        = \begin{cases} 0 & \a \le \t_2 \\  \frac{1}{\cos \t_2-\cos\t_1} & \t_2 < \a <  \t_1 \\ 0 & \a \ge \t_1
      \end{cases}.
$$ 
Hence, it follows readily that 
\begin{align*}
  \int_{\triangle_2} \CM_{2}(\a; \t) \,\d \t \, & = \int_\a^\pi \int_0^\a  \frac{1}{\cos \t_2-\cos\t_1} \,\d \t_2 \,\d \t_1 \\
    & =  \int_\a^\pi \int_0^\a  \frac{1}{2 \sin \frac{\t_1-\t_2}2 \sin \frac{\t_1+\t_2}2 } \,\d \t_2 \,\d \t_1 \\
    & \le \frac{\pi}{\min\{\sin \f{\a}{2}, \cos \f{\a}{2}\}} \int_\a^\pi \int_0^\a \frac{1}{\t_1 - \t_2} \,\d \t_2 \,\d \t_1
\end{align*}
by elementary trigonometric inequalities. The last integral is evidently bounded for $0 < \a < \pi$, so that
$\CM_{1}(\a; \t) \in L^1(\TT^2)$ for $0 < \a < \pi$. 

For $d > 2$, we use induction on $d$ and the already stated recurrence relation from above for B-splines. Since 
\begin{align*} 
 \CM_{d+1}( \alpha| \t) =\, &\frac{\cos \alpha-\cos \t_1}{\cos \t_{d+1}-\cos \t_1}\CM_d(\alpha |\t_1,\ldots,\t_d) \\
       & +\frac{\cos \t_{d+1}-\cos \alpha}{\cos \t_{d+1}-\cos \t_1}\CM_{d}(\alpha |  \t_2,\ldots, \t_{d+1})
\end{align*}
and $\CM_{d+1}( \alpha| \t) =0$ if $\alpha \not \in (\t_1,\t_{d+1})$, it follows that for $\alpha \in [0,\pi]$ and 
$\theta \in \triangle_d$,
$$
\CM_{d+1}( \alpha| \t)\leq \CM_d(\alpha |\t_1,\ldots,\t_d)+\CM_{d}(\alpha |  \t_2,\ldots, \t_{d+1}).
$$
Consequently, the integrability of $\CM_{d+1}(\a | \t)$ follows eventually from the case $d=1$ by induction.
\end{proof}

Since $\CM_d(\a;\cdot)$ is a nonnegative integrable function, we can expand it into multiple Fourier series,
which leads us to consider the Fourier coefficients of the B-spline function as a function of its knots. More 
interestingly, we consider the $\ell^1$-sum of its Fourier coefficients. 

\begin{defn}
For $d \ge 2$ and $n \in \NN_0$, we define
$$
   m_{n,d} (u):= \frac{1}{(2\pi)^d} \int_{\TT^d} M_{d-1} (u | \cos \t_1,\ldots,\cos\t_d) \frac{E_n(\t)}{N_d(n)} \,\d \t. 
$$
\end{defn}

By the definition of $E_n(\t)$, $m_{n,d}$ is the $\ell^1$ mean of the Fourier transform of the B-spline function
$\t \mapsto M_{d-1} (u | \cos \t_1,\ldots,\cos\t_d)$ with respect to its knots. 

\begin{thm}
The family of functions $\{m_{n,d}: n \in \NN_0\}$ and the family $\{h_{n,d}: n \in \NN_0\}$ are biorthogonal; 
more precisely, 
\begin{equation} \label{eq:biortho}
   \int_{-1}^1 m_{n,d}(u) h_{n',d}(u) \,\d u =   \delta_{n,n'}, \qquad n, n' \in \NN_0.
\end{equation}
\end{thm}

\begin{proof}
Multiplying the first identity of \eqref{eq:prod-poisson} by $E_n(\t)$ and integrating over $\t \in \TT^d$, we obtain
$$
   r^n= \frac{1}{(2\pi)^d} \int_{\TT^d}  \frac{(1-r^2)^d} {\prod_{i=1}^d (1-2 r \cos \t_i + r^2)}\frac{E_n(\t)}{N_d(n)}  \,\d \t.
$$ 
Using \eqref{eq: Md-Poisson} and exchanging the order of integrals on the right-hand side, we obtain
\begin{align*}
     r^n  & = \frac{ (d-1)!}{(2\pi)^d} \int_{-1}^1  \frac{(1-r^2)^d}{(1-2 r u + r^2)^d} 
        \int_{\TT^d} M_{d-1}(u | \cos \t_1,\ldots,\cos \t_d) \frac{E_n(\t)}{N_d(n)}   \,\d \t  \,\d u \\
          & = (d-1)! \int_{-1}^1  \frac{(1-r^2)^d}{(1-2 r u + r^2)^d} m_{n,d}(u) \,\d u \\
          & = \sum_{k =0}^\infty \int_{-1}^1 h_{k,d}(u) m_{n,d}(u) \,\d u \, r^k,        
\end{align*}
where the last step follows from \eqref{eq:generat-hnd}. Since the above identity holds for $|r| <1$, comparing
the coefficients of $r^n$ proves \eqref{eq:biortho} by linear independence. 
\end{proof}

Using the orthogonality, we can now derive a series expansion of $m_{n,d}$. Let us start with $d =2$. 

\begin{prop}
For $n =0, 1,2,\ldots$, $m_{n,2}(u) = 0$ if $|u| \ge 1$, and furthermore 
\begin{equation} \label{eq:mnd=2}
  m_{n,2}(\cos \a) = \frac{2}{\pi} \sum_{k=0}^\infty \frac{\sin((n+2k+1) \a)}{n+2k+1}, \qquad 0 < \a < \pi.
\end{equation}
\end{prop}

\begin{proof}
Let $n \ge 0$ be fixed. Since $m_{n,2}(\pm 1) =0$, we assume that $m_{n,2}(u)$ contains a factor $\sqrt{1-u^2}$
and takes  the form
$$
  m_{n,2}(u) =  \frac{2}{\pi}\sqrt{1-u^2} \sum_{k=0}^\infty a_k \frac{U_{n+2k}(u)}{n+2k+1},
$$
where the coefficients $a_k$ are real numbers that will be determined by the biorthogonality of \eqref{eq:biortho} and $U_n$ are as usual the Chebyshev polynomials of the second kind, satisfying $U_n=C^1_n$.
Using the orthogonality of $U_n$,
$$
  \frac2 {\pi} \int_{-1}^1 U_n(t) U_m(t) \sqrt{1-t^2} \,\d t = \delta_{n,m}, \qquad n,m \ge 0,
$$
and the second explicit formula, in \eqref{eq:hnd=}, 
$$
  h_{\ell,2}(u) = (\ell+1)U_\ell(u) - (\ell-1) U_{\ell-2}(u), \qquad \ell \ge 0,
$$
we see that \eqref{eq:biortho} becomes, for $\ell =0,1,\ldots$,
\begin{align*}
  \delta_{\ell,n} = \int_{-1}^1 m_{n,2}(u) h_{\ell,2}(u) \,\d u 
  & =  \sum_{k=0}^\infty a_k  \frac2 {\pi} \int_{-1}^1 \frac{U_{n+2k}(u)}{n+2k+1} h_{\ell,2}(u) \sqrt{1-u^2} \,\d u.
\end{align*}
If $\ell$ and $n$ have different parity, then both the right-hand and the left-hand side are zero. Assume now
that $\ell$ and $n$ have the same parity. If $\ell < n$, then the right-hand side is trivially zero by the orthogonality
of $U_n$. Thus, we only need to consider $\ell = n + 2j$ for $j = 0,1,\ldots$, for which the identity becomes
$$
  \delta_{j,0} = a_j - a_{j-1}, \qquad j =0, 1, 2, \ldots,
$$
where $a_{-1} =0$, so that $a_0 =1$ and $a_j =a_{j-1}$ for $j \ge 0$. Hence, $a_j = 1$ for $j=0,1,\ldots$. 
Now, setting $u = \cos \a$, then $\sqrt{1-u^2} U_{n+2k}(u) = \sin ((n+2k+1) \a)$, which gives the expression
$m_{n,d}$ in \eqref{eq:mnd=2}. 
\end{proof}

It turns out, surprisingly, that the expression \eqref{eq:mnd=2} can be written, for each $n$, as a final sum. 

\begin{thm}\label{thm:mnd=2}
For $n =0, 1,2,\ldots$, and $0 < \a < \pi$,   
\begin{align} 
  m_{2 n,2}(\cos \a) \, & = \frac12 - \frac{2}{\pi} \sum_{k=0}^{n-1} \frac{\sin((2k+1) \a)}{2k+1},  \label{eq:mnd=2Even}\\
  m_{2 n+1,2}(\cos \a)\,& = \frac12 - \frac{\a}{\pi} - \frac{2}{\pi} \sum_{k=1}^{n} \frac{\sin(2k\a)}{2k}. \label{eq:mnd=2Odd}
\end{align}
In particular, we obtain 
\begin{equation} \label{eq:m0d=2}
 m_{0,2}(u) = \frac{1}{(2\pi)^2} \int_{\TT^2} M_{1} (u | \cos \t_1, \cos\t_2) \,\d \t_1 \,\d \t_2 = 
     \begin{cases} 0 &\, u\leq -1, \\  \frac{1}{2} &\, -1< u <1, \\ 0 &\, u\geq1.  \end{cases}
\end{equation}
\end{thm} 

\begin{proof}
Let $f_0$ and $f_1$ be odd $2\pi$ periodic functions so that their restriction on $[0,\pi]$ are defined by 
$$
  f_0(\t) = \begin{cases} 0 & \t = 0, \\ \f\pi 4 &  0< \t <\pi, \\ 0, & \t= \pi, \end{cases} \qquad \hbox{and} \quad 
  f_1(\t) = \begin{cases}  \f \t 2 &  0 \le  \t < \pi, \\ 0, & \t= \pi. \end{cases} 
$$
A quick computation shows that the Fourier series of $f_0$ is given by 
$$
  f_0(\t) = \sum_{k=0}^\infty \frac{\sin (2k+1)\t}{2k+1}, \qquad - \pi \le \t \le \pi,
$$
where the convergence is pointwise. This gives immediately \eqref{eq:m0d=2}. Furthermore, 
for $m_{2n,2}$, we obtain from \eqref{eq:mnd=2} 
$$
 m_{2n,2}(\cos \a) = \f2 \pi \sum_{k= n}^\infty \frac{\sin ( (2k+1)\a)}{(2k+1)}
     = \f 2 \pi \left[\f \pi 4 -  \sum_{k= 0}^{n-1} \frac{\sin ( (2k+1)\a)}{(2k+1)} \right],
$$
which is \eqref{eq:mnd=2Even}. Another quick computation shows that the Fourier series of $f_1(\theta)$ is  
$$
  f_1(\t) = \sum_{n=1}^\infty \frac{( -1)^{n-1} \sin (n \t)}{n} = \sum_{n=0}^\infty \frac{\sin((2n+1)\t)}{2n+1}
       - \sum_{n=1}^\infty \frac{\sin(2n \t)}{2n},
$$ 
which shows in particular, together with the Fourier series of $f_0$, 
$$
\sum_{n=1}^\infty \frac{\sin(2n \t)}{2n} = \frac{\pi} 4 - \f{\t}{2}, \qquad  0 < \t < \pi.
$$
Thus, for $m_{2n+1,2}$, we obtain from \eqref{eq:mnd=2} that 
$$
 m_{2n+1,2}(\cos \a) = \f2 \pi \sum_{k= n+1}^\infty \frac{\sin ( 2k \a)}{2k}
     = \f 2 \pi \left[\frac{\pi} 4 - \f{\a}{2} - \sum_{k= 0}^{n} \frac{\sin ( 2k \a)}{2k} \right],
$$
which is \eqref{eq:mnd=2Odd}. 
\end{proof}

\begin{prop}
For $d >2$, the function $m_{n,d}$ is of the form 
\begin{equation} \label{eq:mnd-series}
   m_{n,d} (u) = (1-u^2)^{d-\frac{3}{2}} \frac{c_{d-1}}{(d-1)!} 
      \sum_{k=0}^\infty  \frac{(d-1)_k}{k!} \frac{C_{n+2k}^{d-1}(u)}{C_{n+2k}^{d-1}(1)},\qquad -1\leq u\leq1,
\end{equation}
where $c_{d-1}$ is the normalization constant defined by its reciprocal 
$$ 
c_{d-1}^{-1} = \int_{-1}^1(1-u^2)^{d-\frac32}\d u = \frac{\Gamma\left(\frac12\right)\Gamma\left(d-\frac12\right)}{(d-1)!}.
$$
\end{prop}

\begin{proof} 
For $d > 2$, the function $u\mapsto M_{d-1}(u | \cos \{\cdot\}_1, \ldots, \cos \{\cdot\}_d)$ has support set in $(-1,1)$ 
and has $d-2$ continuous derivatives, it follows that $m_{n,d}$ is a continuous and $m_{n,d}^{(j)}(\pm 1) = 0$ for $j = 0,1,\ldots,
d-2$. We assume that $m_{n,d}$ has the series expansion
$$
   m_{n,d} (u) = \frac{c_{d-1}}{(d-1)!} (1-u^2)^{d-\frac{3}{2}} \sum_{k=0}^\infty a_k^n \frac{C_{n+2k}^{d-1}(u)}{C_{n+2k}^{d-1}(1)},
$$
where the coefficients $a_k^n$ are to be determined by the biorthogonality, and the choice of index $n+2k$ 
comes from \eqref{eq:biortho} and the parity of $h_{n,d}$. Now, the orthogonality of the Gegenbauer polynomials is 
equivalent to
\begin{equation} \label{eq:GegenOrtho2}
  c_{d-1} \int_{-1}^1 \frac{C_n^{d-1}(u)}{C_n^{d-1}(1)}  Z_m^{d-1}(u) (1-u^2)^{d-\f32}\,\d u = \delta_{n,m}.
\end{equation}
Using the second explicit formula of $h_{n,d}$ in \eqref{eq:hnd=}, which shows that the term of the highest 
degree in $h_{n,d}$ is $(d-1)!Z_n^{d-1}$, whereas the term of the lowest degree in $m_{n,d}$ contains $C_n^{d-1}$,
or the term $k = 0$ in the sum, the orthogonality \eqref{eq:GegenOrtho2} implies that the identity \eqref{eq:biortho} 
for $n' = n$ becomes
\begin{align*}
     a_0^n =  \delta_{n,n} = 1,\quad \forall n\in \NN_0;
\end{align*}
moreover, for $n'=n+2\ell$ and $\ell = 1,2,3,\ldots$,  \eqref{eq:biortho} becomes 
\begin{align*}
  0 = \int_{-1}^1 m_{n,d}(u) h_{n+2\ell,d}(u)\,\d u = (d-1)! \sum_{j=0}^{d-1} (-1)^j \binom{d-1}{j} a_{\ell-j}^n.
\end{align*}
Thus, using $a_j^n = 0$ for $j < 0$ and recalling that $a_0^n = 1$, we see that $a_j^n$ satisfy
\begin{equation} \label{eq:aj-recur}
     \sum_{j=0}^{\ell}  (-1)^j\binom{d-1}{j}a_{\ell-j}^n = \delta_{0,\ell}, \qquad \ell = 0,1,\ldots, d-1.
\end{equation}
It shows, in particular, that $a_j^n$ are independent of $n$ and they can be determined recursively so that
the solution is unique. It turns out that the solution \eqref{eq:aj-recur} is given
explicitly by $a_j =  (d-1)_j / j!$. To verify that this is indeed the case, we use the identity 
$$
  a_{\ell-j} = \f{ (d-1)_{\ell -j}}{(\ell-j)!} = \frac{(d-1)_\ell (-\ell)_j} {\ell!  (2-\ell-d)_j},
$$
which follows from $(-x)_{\ell-j} = (-x)_\ell (-1)^j (1-\ell-x)_j$, and $\binom{d-1}{j} = (-1)^j (-d + 1)_j /j!$ to write 
the right-hand side of \eqref{eq:aj-recur} as a hypergeometric function, 
$$
 \sum_{j=0}^{\ell-1}  (-1)^j\binom{d-1}{j}a_{\ell-j}^n = \frac{(d-1)_\ell}{\ell!} {}_2F_1\left (\begin{matrix} -\ell, \, 1-d \\
 2-d-\ell \end{matrix};1\right )
  = \frac{(d-1)_\ell (1-\ell)_\ell }{\ell!  (2-d-\ell)_\ell},
 $$
where the last step follows from the Chu-Vandermonde identity. Since $(-\ell+1)_\ell = 0$ for $\ell \in \NN$, this 
verifies that $a_j =   (d-1)_j / j!$ is the solution of \eqref{eq:aj-recur}.  
\end{proof}
 
\begin{thm}
For $d > 2$ and $-1 \le u \le 1$, 
\begin{align*}
 m_{0,d}(u) \, & = \frac{1}{(2\pi)^d} \int_{\TT^d} M_{d-1}(u | \cos \t_1,\ldots,\cos \t_d)\,\d \t \\
    &  =  \frac{\Gamma(\frac{d+1}2)}{\sqrt{\pi}\Gamma(\f{d}{2})(d-1)!} (1-u^2)^{\f{d-2}{2}}.
\end{align*}
\end{thm}

\begin{proof}
Let $g(u) = (1-u^2)^{- \f{d-1}{2}}$. We compute the Fourier-Gegenbauer coefficients $\hat g^{d-1}_n$
defined by 
\begin{align*}
  \hat g^{d-1}_n \,& = c_{d-1} \int_{-1}^1 g(t) C_n^{{d-1}}(t) (1-t^2)^{d-\f32}\,\d t \\
     &  = c_{d-1} \int_{-1}^1 C_n^{{d-1}}(t) (1-t^2)^{\frac{d-2}{2}}\,\d t.
\end{align*}
By the parity of $C_n^{d-1}$, $\hat g^{d-1}_n = 0$ if $n$ is odd. To compute $\hat g_{2n}^{d-1}$, we
use the connection coefficients of Gegenbauer polynomials \cite[Theorem 7.1.4', p. 360]{AAR}, which gives
$$
 C_{2n}^{d-1}(t) = \sum_{k=0}^n \frac{(d-1)_{2n-k} (\frac{d-1}{2})_{k} }{(\f{d-1}{2}+1)_{2n-k} k!} 
      Z_{2n-2k}^{\frac{d-1}{2}}(u).
$$
Using the orthogonality of $Z_m^{\frac{d-1}{2}}$, we then conclude that 
$$
   \hat g^{d-1}_n = \frac{c_{d-1}}{c_{\f{d-1}{2}} } \frac{(d-1)_{n} (\frac{d-1}{2})_{n} }{(\f{d-1}{2}+1)_{n} n!} 
   = \frac{c_{d-1}}{c_{\f{d-1}{2}}}  \frac{(d-1)_{n}}{n!} \frac{d-1}{2n+d-1}.
   $$
   (Note that $c_\cdot$ is also defined for a non-integral index.) 
Consequently, the Fourier-Gegenbauer expansion of $g$ is given by
$$
  g (u)= \sum_{n=0}^\infty \hat g_n^{d-1} \frac{C_n^{d-1}(u)}{h_n^{d-1}} 
     = \frac{c_{d-1}}{c_{\f{d-1}{2}}} \sum_{n=0}^\infty  \frac{(d-1)_{n}}{n!}  \frac{C_{2n}^{d-1}(u)}{C_{2n}^{d-1}(1)}, 
$$
where $h_n^\l$ denotes the $L^2([-1,1], (1-t^2)^{\l-\f12})$ norm of $C_n^\l(t)$ and it is equal to 
$$
  h_n^\l = \frac{\l}{n+\l} C_n^\l(1), \qquad n =0,1,2,\ldots.
$$
Comparing with \eqref{eq:mnd-series} with $n =0$, we see that 
\begin{equation}\label{eq:m0d(u)=}
  m_{0,d} = (1-u^2)^{d-\frac{3}{2}} \frac{c_{\f{d-1}2}}{(d-1)!} g(u) = \frac{c_{\f{d-1}2}}{(d-1)!} (1-u^2)^{\f{d-2}2},
\end{equation}
which is the stated result. 
\end{proof}

In the case of $d =2$, it is easy to see that the explicit formula \eqref{eq:mnd=2} implies the relation
$$
    m_{n,2}(u) - m_{n+2,2}(u) = \frac{2}{\pi} \frac{\sin (n+1) \a}{n+1}
      = \frac{2}{\pi} \sqrt{1-u^2} \frac{U_n(u)}{U_n(1)}, \quad u = \cos \a.
$$
The following corollary gives a $d$-dimensional version of this recursive relation for $m_{n,d}$. 

\begin{cor}
For $d \ge 2$, $n=0,1,2,\ldots$, 
\begin{equation}\label{eq:mnd-recur}
(d-1)!  \sum_{j=0}^{d-1} (-1)^j \binom{d-1}{j} m_{n+2j,d}(u)
   = c_{d-1}(1-u^2)^{d-\frac32} \frac{C_n^{d-1}(u)}{C_n^{d-1}(1)}.
\end{equation}
\end{cor}

\begin{proof}
Using the explicit formula of $m_{n,d}$ in \eqref{eq:mnd-series}, we obtain
\begin{align*}
 & (d-1)! \sum_{j=0}^{d-1} (-1)^j \binom{d-1}{j}  m_{n+2j,d} (u) \\
 & \qquad  = (1-u^2)^{d-\frac{3}{2}} c_{d-1} \sum_{j=0}^{d-1} (-1)^j \binom{d-1}{j}
     \sum_{k= j}^\infty  \frac{(d-1)_{k-j}}{(k-j)!} \frac{C_{n+2k}^{d-1}(u)}{C_{n+2k}^{d-1}(1)} \\
 & \qquad  = (1-u^2)^{d-\frac{3}{2}} c_{d-1} \sum_{k=0}^\infty 
  \sum_{j=0}^{d-1} (-1)^j \binom{d-1}{j} \frac{(d-1)_{k-j}}{(k-j)!} \frac{C_{n+2k}^{d-1}(u)}{C_{n+2k}^{d-1}(1)},
\end{align*}
where we have used the convention that $(d-1)_{k-j} = 0$ if $j > k$ and $\binom{d-1}{j} =0$ if $j > d-1$. 
The stated result then follows from \eqref{eq:aj-recur}. 
\end{proof}

In particular, the identity \eqref{eq:mnd-recur} shows that the finite combination of $m_{n,d}$ in the
left-hand is a polynomial of degree $n$ multiplied by $(1-t^2)^{d-\f32}$. The recursive formula can be
used to determine $m_{n,d}$ if we know the first $d-1$ elements $m_{0,d}, \ldots, m_{d-1,d}$. 
However, notice that the explicit expression of $m_{0,d}$ in \eqref{eq:m0d(u)=} contains the factor
$(1-u^2)^{\frac{d-2}{2}}$, which has a power different from the power $d - \f32$ in the right-hand side of 
\eqref{eq:mnd-recur}, we see that an analog of \eqref{eq:mnd=2Even} is unlikely to hold for $d>2$;
in particular, $m_{n,d}$ will not be a polynomial when $n$ is even for $d >2$. 

\section{Positive definite $\ell^1$-invariant functions}
\setcounter{equation}{0}
 
A function $f\in C(\TT^d)$ is a positive definite function (PDF) if for every $\Xi_N = \{\Theta_1,\ldots, \Theta_N\}$ of 
pairwise distinct points in $\TT^d$ and $N \in \NN_0$, the matrix 
$$
         f[\Xi_N]= \left [ f(\Theta_i - \Theta_j) \right]_{i,j = 1}^N
$$ 
is nonnegative definite; it is called a strictly positive definite function (SPDF) if the matrix is positive definite. 
Let $\Phi(\TT^d)$ denote the set of PDFs on $\TT^d$. By the definition of PDF, the
space $\Phi(\TT^d)$ is closed under linear combination with nonnegative coefficients; that is, if $f, g \in 
\Phi(\TT^d)$ and $c_i$ and nonnegative constants, then $c_1 f + c_2 g \in \Phi(\TT^d)$. The PDFs on $\TT^d$ 
are characterized by the following theorem

\begin{thm}
A function $f \in C(\TT^d)$ is a PDF on $\TT^d$ if and only if the Fourier coefficients $\hat f_\a$ are 
nonnegative for all $\a \in \NN_0^d$. 
\end{thm}

One direction of the theorem follows readily from the closeness of $\Phi(\TT^d)$ and that the exponential
functions $\e^{\i \a \cdot x} \in \Phi(\TT^d)$ for all $\a \in \NN_0^d$, since 
$$
  \sum_{i=1}^N \sum_{j=1}^N  c_i c_j\e^{\i \a \cdot (\Theta_i - \Theta_j)} =  \left | \sum_{i=1}^N c_i \e^{\i \a \cdot \Theta_i }
   \right|^2 \ge 0.
$$
In the other direction, if $f$ is PDF on $\TT^d$ then by the periodicity of $f$, 
$$
   \int_{\TT^d}  f (\theta) \,\d \theta = \frac{1}{(2 \pi)^d} \int_{\TT^d} \int_{\TT^d}  f (\theta - \phi ) \,\d \theta \,\d \phi. 
$$
The right-hand side is nonnegative if $f$ is a trigonometric polynomial, as can be seen by applying a positive 
cubature rule for the integral over $\TT^d$ and using the positive definiteness of $f$. In particular, this shows
that the left-hand side integral is nonnegative. Since $\e^{\i \a \cdot \t}$ is PDF, it follows from Schur's theorem that
$$
    \hat f_\a = \frac{1}{(2 \pi)^d} \int_{\TT^d} f(\t) \e^{-\i \a\cdot \t} \,\d \t  \ge 0, \qquad \forall \a \in \NN_0^d. 
$$

Recall that a function $f: \TT^d \to  \RR$ is $\ell^1$-invariant if $\hat f_\a = \hat f_\b$ for all $|\a| = |\b|$, 
$\a,\b \in\NN_0^d$. These functions are given as follows. 

\begin{thm}
A function $f \in L^1(\TT^d)$ is $\ell^1$-invariant if and only if 
\begin{equation} \label{eq:f=divideFd}
     f (\t_1,\ldots,\t_d) = [\cos \t_1, \ldots, \cos \t_d] F_d,
\end{equation}
where $F_d: [-1,1]\to  \RR$ is a $(d-1)$-times differentiable function and satisfies
\begin{equation}\label{eq:Fd=}
  F_d(\cos \t) = 2 (-1)^{\lfloor \f{d+1}{2}\rfloor} (\sin \t)^{d-1} \sum_{n=0}^\infty a_n
      \begin{cases}  - \sin (n \t) & \hbox{for $d$ even}, \\   \cos (n \t) & \hbox{for $d$ odd}, \end{cases}  
\end{equation}
with a sequence of real numbers $\{a_n\}_{n\ge 0}$. 
\end{thm} 

\begin{proof}
If $f$ is the given divided difference of $F_d$ given in \eqref{eq:Fd=}, then 
$$
  f(\t) =   \sum_{n=0}^\infty a_n [\cos \t_1,\ldots, \cos \t_d] H_{n,d} = 
       \frac12 \sum_{n=0}^\infty a_n E_n(\t). 
$$ 
By the definition of $E_n(\t)$, it follows readily that $\hat f_\a = \hat f_\b$ if $|\a| = |\b|$. Moreover, 
when its knots coalesce, the divided difference becomes a derivative as we have mentioned before. 
By \eqref{eq:divder} and the $(d-1)$-th differentiability of $F$, $f$ is continuous. 

In the other direction, if $f$ is $\ell^1$-invariant, then its Fourier series is given by \eqref{eq:l1-function}. 
By \eqref{eq:En-hnd}, we obtain that $f (\t_1,\ldots,\t_d) = [\cos \t_1, \ldots, \cos \t_d] F$ with $F$ given by
$$
  F(u) = \sum_{n=0}^\infty \hat f_n H_{n,d}(u),
$$
so that $F$ is of the form \eqref{eq:Fd=} with $a_n = \wh f_n$ by \eqref{eq:Hnd=}. The continuity of
 $f$ requires that $F$ has continuous derivatives of $(d-1)$ order by \eqref{eq:divder}. 
\end{proof}

If $f$ is $\ell^1$-invariant, then $\hat f_\a = \hat f_{|\a|}$, so that $\hat f_\a \ge 0$ if and only if 
$\hat f_n \ge 0$ in \eqref{eq:l1-function}. Consequently, the following characterization of PDFs 
holds. 

\begin{thm}
Let $f \in C(\TT^d)$ be $\ell^1$-invariant. Then $f$ is a PDF on $\TT^d$ if and only if $f$ is given by 
\eqref{eq:f=divideFd} and \eqref{eq:Fd=} with $\hat f_n \ge 0$ for all $n \in \NN_0$. 
\end{thm}

The SPDFs have been characterized in \cite[Theorem 1]{Guella}, where it is proved that a PDF function is also
SPDF if and only if the set of indices $\mathcal{G}=\lbrace \alpha \in \ZZ^{d} \mid \hat{f}_{\alpha}>0 \rbrace$ 
intersects all the translations of each subgroup of $\ZZ^{d}$ that has the form 
$$
 (a_1\ZZ,a_2\ZZ,\ldots,a_d \ZZ),\quad  a_1,\a_2,\ldots,a_d \in \NN.
$$
More precisely, for every pair of vectors $\gamma \in \ZZ^{d},\beta \in \NN^{d}$ there exists a $z\in \ZZ^{d}$ with 
$\hat{f}_{\alpha}>0$ and $\alpha_j=\gamma_j+z_j \beta_j$, $j=1,\ldots,d$. For $\ell^1$-invariant functions, the
SPDFs are characterized below. 

\begin{thm}
Let $f \in C(\TT^d)$ be $\ell^1$-invariant. Then $f$ is a SPDF on $\TT^d$ if and only if $f$ is given by 
\eqref{eq:f=divideFd} and \eqref{eq:Fd=}, $f$ is PDF and for any pair $n<\ell \in \NN$ there is an
$m\in \NN$ with $\hat{f}_{n+m\ell}>0$ or $\hat{f}_{(\ell-n)+m\ell}>0$.  
\end{thm}

\begin{proof}
We prove that the given condition is equivalent to the condition given in \cite{Guella}, if we assume the kernel 
to be $\ell^1$-summable. 

We start with sufficiency. Assume $f$ is PDF and satisfies the condition of the theorem. For any two vectors $\gamma \in \ZZ^{d},\beta \in \NN^{d}$ we can assume without loss of generality that $0\leq \gamma_j\leq \beta_j$, $j=1,\ldots d$. 
We define $n=\vert \gamma \vert$ and $\ell=\vert \beta \vert$. Then there exists an $m\in \NN_0$ with 
$\hat{f}_{n+\ell m}>0$ or $\hat{f}_{(\ell-n)+\ell m}>0$.  If $\hat{f}_{n+\ell m}>0$,  then
$$
\hat{f}_{\gamma+m \beta}=\hat{f}_{n+\ell m}>0
$$
since $\vert \gamma+m \beta\vert= n+ m\ell$. Whereas if $\hat{f}_{(n-\ell)+\ell m}>0$, we can choose $m'=m+1$. 
Now $\vert \gamma-m' \beta \vert=\vert \gamma-\beta- m \beta \vert =(\ell-n) +m \ell$ implies 
$$
 \hat{f}_{\gamma-m'\beta}=\hat{f}_{(n-\ell)+\ell m}>0.
 $$

For the necessity, assume the assumption of the theorem does not hold, then there exists $\ell>n \in \NN$ 
such that $\hat{f}_{n+m\ell}=0$ and $\hat{f}_{(\ell-n)+\ell m}=0$ for all $m\in \NN$. Set  
$$
\gamma=(n,0,\ldots,0)\in \ZZ^{d},\quad \beta=(\ell,\ldots,\ell)\in \ZZ^{d}.
$$ 
Then for any $z\in \ZZ^{d}$ define $\alpha\in \ZZ^{d}$, $\alpha_j=\gamma_j+z_j \beta_j$, so that with $m=\vert z\vert$,
\begin{equation*}
\vert \alpha\vert=\vert n+ z_1 \ell\vert + \ell \sum_{j=2}^{d-1} \vert z_j\vert= \begin{cases} n+m\ell,& z_1\geq0, \\
(\ell-n)+(m-1)\ell,& z_1<0. \end{cases} 
\end{equation*}
Thereby, there would be no coefficients $\hat{f}_{\alpha}>0$, $\alpha_j=\gamma_j+z_j \beta_j$ for this choice 
 of $\gamma$ and $\beta$, contradicting strict positive definiteness. 
\end{proof}


\begin{thebibliography}{99}

\bibitem{AAR}
       G. E. Andrew, R. Askey, and R. Roy, 
       \textit{Special Functions}, 
        Encyclopedia of Mathematics and its Applications \textbf{71},
        Cambridge University Press, Cambridge, 1999. 

\bibitem{BX1}
        H. Berens and Y. Xu,
        Fej\`er means for multivariate Fourier series.
        \textit{Math. {{Z.}}}  \textbf{221} (1996), 449--465.

\bibitem{BX2}
        H. Berens and Y. Xu, $\ell$-1 summability for multivariate Fourier integrals and positivity. 
        \textit{Math. Proc. Cambridge Phil. Soc.} \textbf{122} (1997), 149--172. 
        
\bibitem{Guella}
        J. Guella and V. A. Menegatto, 
        Strictly positive definite kernels on the torus,
        \textit{Constr. Approx.} \textbf{46} (2017), 271--284.

\bibitem{N}
        Z. N\'emeth,
        On multivariate de la Vall\'ee Poussin-type projection operators.
        \textit{J. Approx. Theory} \textbf{186} (2014), 12--27. 

\bibitem{MJDP}
        M.J.D. Powell, \textit{Approximation theory and methods}.
        Cambridge University Press. 1982.

        
\bibitem{Sz} 
       G. Szeg\H{o}, 
       \textit{Orthogonal polynomials}. 4th edition, Amer. Math. Soc., Providence, RI. 1975.

\bibitem{SV} 
       L. Szili and P. V\'ertesi,
       On multivariate projection operators. 
       \textit{J. Approx. Theory} \textbf{159} (2009), 154--164. 
       
\bibitem{FW1}
        F. Weisz, 
        $\ell^1$-summability of higher-dimensional Fourier series. 
        \textit{J. Approx. Theory} \textbf{163} (2011), 99--116. 

\bibitem{FW2}
        F. Weisz,
        Summability of multi-dimensional trigonometric Fourier series. 
        \textit{Surv. Approx. Theory} \textbf{7} (2012), 1--179. 
                
\bibitem{FW3}
        F. Weisz,
        Lebesgue points of $\ell^1$-Ces\`aro summability of $d$-dimensional Fourier series. 
        \textit{Adv. Oper. Theory} \textbf{6} (2021), no. 3, Paper No. 48, 24 pp. 
 
\bibitem{X95}
       Y. Xu,
       Christoffel functions and Fourier Series for multivariate orthogonal polynomials. 
       \textit{J. Approx. Theory}, \textbf{82} (1995), 205--239.

\end{thebibliography}
\end{document}